\documentclass{amsart}
\usepackage{amsmath,amssymb}
\usepackage[all]{xy}

\theoremstyle{plain}

\newtheorem{theorem}{Theorem}[section]
\newtheorem{proposition}[theorem]{Proposition}
\newtheorem{lemma}[theorem]{Lemma}
\newtheorem{corollary}[theorem]{Corollary}

\newtheorem{theorem*}{Theorem}[]

\theoremstyle{definition}
\newtheorem{definition}[theorem]{Definition}

\newtheorem{examples}[theorem]{Examples}

\theoremstyle{remark}
\newtheorem{remark}[theorem]{Remark}

\newcommand{\secref}[1]{Section~\ref{#1}}
\newcommand{\thmref}[1]{Theorem~\ref{#1}}
\newcommand{\propref}[1]{Proposition~\ref{#1}}
\newcommand{\lemref}[1]{Lemma~\ref{#1}}

\newcommand{\exsref}[1]{Examples~\ref{#1}}

\newcommand{\remref}[1]{Remark~\ref{#1}}

\def\:{{\colon}}

\def\N{{\mathbb N}}

\def\map{\mathrm{Map}}

\def\cat0{\mathrm{cat}_0}

\def\cat{{\mathrm {cat}}}
\def\secat{{\mathrm {secat}}}
\def\TC{{\mathrm {TC}}}

\def\sh{\mathrm{sh}}
\def\relcat{{\mathrm {relcat}}}

\begin{document}

\title[Topological Complexity]
{Topological Complexity of H-Spaces}

\author{Gregory Lupton}
\address{Department of Mathematics,
      Cleveland State University,
      Cleveland OH 44115 U.S.A.}
\email{G.Lupton@csuohio.edu}

\author{J{\'e}r\^ome Scherer}
\address{SB MATHGEOM, MA B3 455,
Station 8, EPFL,
CH - 1015 Lausanne }
\email{jerome.scherer@epfl.ch}
\date{\today}



\thanks{The first-named author acknowledges the hospitality and support of EPFL, and the support of the Cleveland State University FRD grant program. The second-named author is partially supported by FEDER/MEC grant MTM2010-20692.  Both authors acknowledge the support of the Swiss National Science Foundation (project IZK0Z2\_133237).}

\begin{abstract}
Let $X$ be a (not-necessarily homotopy-associative) H-space.  We show that $\TC_{n+1}(X) = \cat(X^n)$, for $n \geq 1$, where $\TC_{n+1}(-)$ denotes the so-called higher topological complexity introduced by Rudyak, and $\cat(-)$ denotes the Lusternik-Schnirelmann category.  We also generalize this equality to an inequality, which gives an upper bound for $\TC_{n+1}(X)$, in the setting of a space $Y$ acting on $X$.
\end{abstract}

\maketitle

\section{Introduction}\label{sec:intro}

The notion of topological complexity, introduced by Farber in \cite{Far03},  is motivated by the study of motion planning algorithms in the field of topological robotics.  See  also \cite{Far04, Far06} for good introductions to the topic and its context.
In this paper we prove results that allow for effective computation of topological complexity, and of higher analogues of the same, in important cases.   

We begin with some notation and basic definitions.
We use $\cat(X)$ to denote the Lusternik-Schnirelmann (L-S) category of $X$ (normalised, so that $\cat(S^n) = 1$).  That is, $\cat(X)$ is the smallest $n$ for which there is an open covering $\{ U_0, \ldots, U_n \}$ by $(n+1)$ open sets, each of which is contractible in $X$.  See \cite{CLOT03} for a general introduction to L-S category and related topics.  
  The \emph{sectional category} of a fibration $p \colon E \to B$, denoted by $\secat(p)$, is the smallest number $n$ for which there is an open covering $\{ U_0, \ldots, U_n \}$ of $B$ by $(n+1)$ open sets, for each of which there is a local section $s_i \colon U_i \to E$ of  $p$, so that $p\circ s_i = j_i \colon U_i \to B$, where $j_i$ denotes the inclusion.  This notion generalizes that of L-S category: if $E$ is contractible, then we have $\secat(p) = \cat(B)$; in general we have $\secat(p) \leq \cat(B)$ (see \cite[Prop.9.14]{CLOT03}). 
Now let $PX$ denote the space of (free) paths on a space $X$: thus we have $PX = \map(I, X)$.  There is a fibration $P_2 \colon PX \to X\times X$, which evaluates a path at initial and final point: for $\alpha \in PX$, we have $P_2(\alpha) = \big(\alpha(0), \alpha(1)\big)$.  This is a fibrational substitute for the diagonal map $\Delta \colon X \to X \times X$.  We define the \emph{topological complexity} $\TC(X)$ of $X$ to be the sectional category $\secat\big( P_2\big)$ of this fibration.  That is, $\TC(X)$ is the smallest number $n$ for which there is an open cover $\{ U_0, \ldots, U_n \}$ of $X \times X$ by $(n+1)$ open sets, for each of which there is a local section $s_i \colon U_i \to PX$ of  $P_2$, i.e., for which $P_2\circ s_i = j_i \colon U_i \to X \times X$, where $j_i$ denotes the inclusion.

We indicate briefly the motivation for topological complexity mentioned above; for a full discussion see \cite{Far03, Far04, Far06}.  A point in $X \times X$ is a pair $(x,y)$ of points in $X$.  We imagine $X$ as the space of all states of some system, and $x$ and $y$ as initial and final states of the system.  Then a section $s\colon X\times X \to PX$ of $P_2\colon PX \to X \times X$ may be viewed as a continuous, deterministic association of a path in $X$, from initial state $x$ to final state $y$, to each pair of states $(x,y)$: a \emph{continuous motion planner}. As is easily shown, however, a global section is possible exactly when $X$ is contractible.  Since one is naturally interested in state spaces that have interesting topology, i.e., which are non-contractible, one is necessarily forced into considering \emph{discontinuous motion planners}, which correspond to a covering and local sections exactly as in our definition for $\TC(X)$.     

We also consider the ``higher analogues"  of topological complexity introduced by Rudyak in \cite{Rud10} (see also \cite{Rud10b} and \cite{BGRT10}).  For this notion, let $n \geq 2$ and consider the fibration
$$P_n \colon PX \to X \times \cdots \times X = X^n,$$
defined by dividing the unit interval $I = [0, 1]$ into $(n-1)$ subintervals of equal length, with $n$ subdivision points $t_0 = 0, t_1 = 1/(n-1), \ldots, t_{n-1} = 1$ (thus $(n-2)$ subdivision points interior to the interval), and then evaluating at each of the $n$ subdivision points, thus:
$$P_n(\alpha) = \big(  \alpha(0), \alpha(t_1), \ldots, \alpha(t_{n-2}), \alpha(1)\big),$$
for $\alpha \in PX$.  This is a fibrational substitute for the $n$-fold diagonal $\Delta_n\colon X \to X^n$.   Then the \emph{higher topological complexity} $\TC_n(X)$ is defined as $\TC_n(X) = \secat(P_n)$.  

The ordinary topological complexity, then, corresponds to the case in which $n = 2$: $\TC_2(X) = \TC(X)$. Higher topological complexity may readily be motivated in terms of motion planners.  A point in $X^n$ is an $n$-tuple $(x, x_2, \ldots, x_{n-1}, y)$.  A section (either local or global) of $P_n$ corresponds to a deterministic association of a path in the state space $X$, from initial state $x$ to final state $y$, and passing through the specified intermediate states $x_2, \dots, x_{n-1}$ in that order.  

Our first main result is as follows (see \secref{sec:TC}):

\begin{theorem*}\label{introthm: TC H-space}
Let $X$ be a connected, CW $H$-space.  Then, for $n \geq 1$,  we have $\TC_{n+1}(X) = \cat(X^n)$.  In particular, we have $\TC(X) = \cat(X)$. 
\end{theorem*}

Since the value of $\cat(X^n)$ is known, or can be determined, in many cases in which $X$ is a Lie group or a more general H-space, this result provides many cases in which the (higher) topological complexity may be determined exactly. This is of great value, since $\TC(X)$, like $\cat(X)$, proves to be a very delicate invariant to compute and its exact value is known in comparatively few cases.

In the case in which $G$ is a topological group, the equality $\mathrm{TC}(G) = \mathrm{cat}(G)$ appears as 
\cite[Lem.8.2]{Far04} and its generalization $\TC_{n+1}(G) = \cat(G^n)$ appears as \cite[Th.3.5]{BGRT10}.  The proofs of both results make essential use of associativity. \thmref{introthm: TC H-space} extends these results to the more general case in which $X$ is an H-space, with no assumption on associativity.

Our second main result is the following (see \secref{sec: Homotopy Transitive}):   

\begin{theorem*}\label{introthm: TC secat shear}
Let $A\colon X\times Y \to X$ be an action of $Y$ on $X$ with strict unit. 
For $n \geq 1$ let $\sh_{n+1} \colon X \times Y^n \to X^{n+1}$ denote the shear map defined by 
$$\sh_{n+1} = \big( p^1, A\circ(p^1, p^2), \ldots, A\circ(p^1, p^{n+1})  \big).$$
Here, $p^j$ denotes projection on the $j$th factor: $p^1 \colon X \times Y^n \to X$ projects onto $X$, but for $j = 2, \ldots, n+1$, $p^j \colon X \times Y^n \to Y$ projects onto $Y$. Then we have
$$\TC_{n+1}(X) \leq \big(\secat(\sh_{n+1}) + 1\big) \big(\cat(Y^n) + 1\big) - 1.$$
In particular, if $n = 1$, then we have $\TC(X) \leq  (\secat(\sh) + 1) (\cat(Y) + 1) - 1$, where $\sh\colon X \times Y \to X \times X$ is defined as $\sh = (p^1, A)$. 
\end{theorem*}

This gives a general upper bound for the (higher) topological complexity that may be applied in many situations of interest and that, in principle, may improve on other general upper bounds in those situations to which it applies.  \thmref{introthm: TC secat shear} is a direct generalization of \thmref{introthm: TC H-space}: if the action $A\colon X\times Y \to X$ is given by an H-space $X$ acting on itself by the multiplication, then \thmref{introthm: TC secat shear} reduces to \thmref{introthm: TC H-space}.

We finish this introduction with a brief outline of the paper. \secref{sec:notation} is a short section in which we review some standard facts concerning sectional category used in the sequel.  In \secref{sec:TC} we give basic inequalities concerning $\TC_{n+1}(X)$, and the proof of \thmref{introthm: TC H-space}.  We give several examples to illustrate the applicability of \thmref{introthm: TC H-space}.  In \secref{sec: Homotopy Transitive} we develop the notion of a shear map coming from an action of one space on another, and give the proof of \thmref{introthm: TC secat shear}.  The final section, \secref{sec:relcat}, is a short, informal section in which we give several interesting questions that stem from our results and examples.

\subsection*{Acknowledgements}  We thank John Oprea for many fruitful discussions on this topic and, in particular, for first showing us the argument on which our results are based.     

%%%%%%%%%%%%%%%%%%%%%%%%%%%%%%%%
\section{Notational  conventions and basic facts}
\label{sec:notation}
%%%%%%%%%%%%%%%%%%%%%%%%%%%%%%%%
Our conventions are that $\cat$, respectively $\secat$ (hence $\TC$), is \emph{one less than} the number of suitable open sets in a covering of $X$, respectively, of $B$. Thus, a space $X$ is contractible exactly when $\cat(X) = 0$ and we have $\cat(S^n) = 1$ for any $n \geq 1$.  Furthermore, a fibration $p \colon E \to B$ admits a section exactly when $\secat(p) = 0$ (and if $E$ is also contractible we then recover the case $\cat(B)=0$). These conventions agree with those used in \cite{BGRT10, CLOT03, MR1802847}, for example.  They entail that $\TC(S^n) = 1$ for $n$ odd, and 
$\TC(S^n) = 2$ for $n$ even. However, some authors use the ``un-normalized" version of $\secat$ (hence $\TC$), whether or not the ``normalized" $\cat$ is used.  Hence, our $\TC(X)$ will be one less than that of some authors, such as Farber in \cite{Far04}.

There are two well-known and basic facts which we use throughout the article to compare sectional categories of various maps.  We state them here; each is easily justified directly from the definitions. 
First, suppose given a fibration $p \colon E \to B$ and any map $f \colon B' \to B$.  Form the  pullback
$$\xymatrix{ Q \ar[d]_{q} \ar[r] & E \ar[d]^{p}\\
B' \ar[r]_{f} & B. }$$
Then we have $\secat(q) \leq \secat(p)$. 
Second, suppose given a fibration $p' \colon E' \to B'$ and any maps $\theta \colon Q \to E'$ and $q\colon Q \to B'$ for which the diagram
$$\xymatrix{ Q \ar[rd]_{q} \ar[rr]^{\theta} & & E' \ar[ld]^{p'}\\
&  B'}$$
is homotopy commutative.  Then we have $\secat(p') \leq \secat(q)$.  In this latter situation, a local section $s\colon B' \to Q$ of $q\colon Q \to B'$ leads initially to a local homotopy section $\theta\circ s\colon B' \to E'$ of $p'\colon E' \to B'$, which may then be adjusted into a section using the assumption that $p'\colon E' \to B'$ is a fibration.

A standard application of the pull-back square situation is as follows.  Suppose that $X$ is a based space with basepoint $x_0$. 
Write $\mathcal{P}X$ for the \emph{based} path space of maps $\alpha\colon I \to X$ with $\alpha(0) = x_0$, and let $\mathcal{P}_1\colon \mathcal{P}X \to X$ denote the evaluation map $\mathcal{P}_1(\alpha) = \alpha(1)$.
The pull-back square
\[
\xymatrix{ \mathcal{P}X \ar[d]_{\mathcal{P}_1} \ar[r] & PX \ar[d]^{P_{2}}\\
X \ar[r]_-{i_2} & X \times X}
\]
yields the inequality $\cat(X) \leq \TC(X)$, since $\cat(X) = \secat(\mathcal{P}_1)$ ($\mathcal{P}X$ is contractible). This can easily be generalized to an inequality concerning
the higher topological complexity. Write $I_2\colon X^n \to X^{n+1}$ for the inclusion $I_2(x_1, \dots, x_n) = (x_0, x_1, \dots, x_n)$, 
with $x_0 \in X$ the basepoint, and $\mathcal{P}_n \colon \mathcal{P} X \rightarrow X^n$ for the map with coordinates given by evaluation at $1/n, 2/n, \dots, 1$.
Although (for $n \geq 2$) this is not the usual based path fibration, nonetheless the space $\mathcal{P}X$ is still 
contractible so that $\secat(\mathcal{P}_n) = \cat(X^n)$.

\begin{lemma}\label{lem: TCn+1 lower bound}
Let $X$ be any connected pointed space. Then $\cat(X^n) \leq \TC_{n+1}(X)$ for any $n \geq 1$.
\end{lemma}

\begin{proof}
Note that we have a pullback square
\[
\xymatrix{ \mathcal{P}X \ar[d]_{\mathcal{P}_n} \ar[r] & PX \ar[d]^{P_{n+1}}\\
X^n \ar[r]_-{I_2} & X^{n+1},}
\]
which gives $\cat(X^n) \leq \secat(P_{n+1})$ from the pullback property of $\secat$ recalled above.  
This latter is $\TC_{n+1}(X)$ by definition, and the inequality follows.  
\end{proof}

Finally, we will sometimes refer to the sectional category of an arbitrary map $f: X \rightarrow Y$ (not necessarily a fibration).
What we mean by that is that we first replace $f$ by a fibration. The sectional category does not depend on the choice of such a fibrational substitute. This allows us to deal with homotopy pull-back squares instead of pulling back a fibration. In particular we
may write equalities such as $\cat(X) = \secat( * \rightarrow X)$.

%%%%%%%%%%%%%%%%%%%%%%%%%%%%%%%%
\section{(Higher) Topological Complexity and H-Spaces}
\label{sec:TC}
%%%%%%%%%%%%%%%%%%%%%%%%%%%%%%%%
First we observe that, for $n \geq 1$, we have the following basic inequalities, which extend in a natural way those for $\TC(X)$ recalled in \secref{sec:intro}.

\begin{proposition}
Suppose $X$ is a connected, non-contractible CW complex.  For $n\geq 1$, we have
$$n \leq \cat(X^n) \leq \TC_{n+1}(X) \leq \cat(X^{n+1}) \leq (n+1) \cat(X).$$
\end{proposition}

\begin{proof}
The first and last of these are standard inequalities concerning $\cat(X)$ (see \cite[Th.1.36, Th.1.46]{CLOT03}; the last, which derives from the product inequality,  needs some separation hypothesis on $X$, which we have covered here by assuming $X$ is CW).  The second inequality is \lemref{lem: TCn+1 lower bound}, and the third follows from the definition of $\TC_{n+1}(X)$ together with the general inequality $\secat(p) \leq \cat(B)$ for any fibration $p\colon E \to B$. 
\end{proof}

\begin{remark}
In \cite{Rud10}, the inequality $\cat(X^n) \leq \TC_{n+1}(X)$ appears to have been overlooked.  It is taken up in \cite[Prop.3.1]{BGRT10}, although by an apparently more involved argument than that used here.
\end{remark}

We now prove our first main result enunciated in \secref{sec:intro}. 

\begin{proof}[Proof of \thmref{introthm: TC H-space}] 
Let $m \colon X \times X \to X$ denote the multiplication.  Without loss of generality we may assume the unit is strict, i.e., $m(x, x_0) = x$ and $m(x_0, x) = x$ for $x \in X$, where $x_0 \in X$ denotes the basepoint.  As is well-known by work of James \cite{MR0133132}, for any CW complex $A$, the set of pointed homotopy classes $[A,X]$ is an algebraic loop.  In particular, the equation $a\cdot z = b$, for $a, b \in [A,X]$, has a unique solution $z\in [A,X]$.  Here and in the remainder of this proof, we use ``$\cdot$" to denote the product in $[A , X]$ induced by the multiplication. Write $\pi_1, \pi_2 \colon X \times X \to X$ for the projections.  In $[X \times X, X]$, we take $a = [\pi_1]$ and $b = [\pi_2]$, and we set $D\colon X \times X \to X$ to be the unique solution to the equation $[\pi_1]\cdot [D] = [\pi_2]$.  That is, we have $\pi_1 \cdot D \sim \pi_2$.  (If $X$ were homotopy-associative with inverse, then here we could take $D(x,y) = x^{-1}y$.)      

Now use the map $D \colon X \times X \to X$ to define a map
$f_n \colon X^{n+1} \to X^n$
as $f_n(x, x_1, \ldots, x_n) = \big( D(x, x_1), \ldots, D(x, x_n) \big)$.   
Form the pullback
\begin{equation}\label{eq: square}
\xymatrix{ Q \ar[d]_{q} \ar[r] & \mathcal{P}X \ar[d]^{\mathcal{P}_n}\\
X^{n+1} \ar[r]_{f_n} & X^n, }
\end{equation}
so that we have 
$$Q = \{ (x, x_1, \ldots, x_n, \alpha) \in X^{n+1} \times \mathcal{P}X \mid \alpha(0) = x_0, \alpha(t_i) = D(x,x_i), i=1, \ldots, n  \},$$
and $q \colon Q \to X^{n+1}$  is just projection on the first $(n+1)$ coordinates.  
Finally, define a map $\theta\colon Q \to PX$,
by setting $\theta\big( (x, x_1, \ldots, x_n, \alpha) \big) = m(x, \alpha)$, which gives a path in $X$ evaluated as 
$m(x, \alpha)(t) = m\big(x, \alpha(t) \big)$ for each $t \in I$. By definition, we have 
$$P_{n+1}\circ \theta\big( (x, x_1, \ldots, x_n, \alpha) \big) = \big( m(x, x_0), m(x, D(x,x_1)), \ldots, m(x, D(x,x_n))  \big).$$
The first coordinate on the right-hand side here is $x$, since we are assuming a strict unit.    Now write $p_i\colon X^{n+1} \to X$ for projection on the $i$th coordinate, and write $p_{1,j}\colon X^{n+1} \to X^2$ for the projection $p_{1,j} = (p_1, p_{j+1})$, so that  $p_{1,j}(x, x_1, \ldots, x_n) = (x, x_j)$.  Then for $j = 1, \dots, n$, we may write the $(j+1)$st coordinate of the above expression as $m(x, D(x, x_j)) = (\pi_1\cdot D)\circ p_{1,j}\big((x, x_1, \ldots, x_n)\big)$.  Bearing in mind that $\pi_1\cdot D \sim \pi_2$, we compute that
\begin{align*}
P_{n+1}\circ \theta &= \big( p_1,  (\pi_1\cdot D)\circ p_{1,1}, (\pi_1\cdot D)\circ p_{1,2}, \ldots, (\pi_1\cdot D)\circ p_{1,n}\big) \circ q\\
&\sim \big(  p_1, \pi_2\circ p_{1,1}, \pi_2\circ p_{1,2}, \ldots, \pi_2\circ p_{1,n} \big) \circ q\\
&= \big(  p_1, p_2, p_3, \ldots, p_{n+1} \big) \circ q = q.
\end{align*} 
Thus we have a homotopy commutative diagram
\begin{equation}\label{eq: triangle}
\xymatrix{ Q \ar[rd]_{q} \ar[rr]^{\theta} & & PX \ar[ld]^{P_{n+1}}\\
&  X^{n+1}.}
\end{equation}
From diagrams (\ref{eq: square}) and (\ref{eq: triangle}) it follows that $\secat(P_{n+1}) \leq \secat(q) \leq \secat(\mathcal{P}_n)$.  The first of these is $\TC_{n+1}(X)$ and the last of these is $\cat(X^n)$ (since $\mathcal{P}X$ is contractible).  Thus, for $X$ an H-space we have $\TC_{n+1}(X) \leq \cat(X^n)$, and equality now follows from the general inequality $\cat(X^n) \leq \TC_{n+1}(X)$.   
\end{proof}

\begin{examples}\label{exs: TC H-space}
\begin{enumerate}
\item \thmref{introthm: TC H-space} applies to any compact Lie group, as was previously known. Thus we have, for instance,  $\TC\big(Sp(3)\big) = \cat\big(Sp(3)\big) = 5$ as computed
in \cite{MR2022385}, and $\TC\big(Sp(9)\big) = 8$ as follows from \cite{MR2272137}.   
\item Let $U(r)$ and $SU(r)$ denote the unitary group and special unitary group, respectively.  It is known that $\cat\big( U(r)\big) = r$ and $\cat\big( SU(r)\big) = r-1$ (see \cite{Sin75} or \cite[Prop.9.5, Th.9.47]{CLOT03}).  Although $\cat(-)$ generally behaves sub-additively with respect to products, we may see that $\cat\big(\big(U(r)\big)^n \big) = nr$ and $\cat\big(\big(SU(r)\big)^n \big) = n(r-1)$.  This is because, \emph{in these cases}, we have a binding lower bound for $\cat(-)$ given by (rational) cohomology, which equals the upper bound given by the product inequality.  Thus it follows that we have $\TC_{n+1}\big(U(r)\big) = nr$ and $\TC_{n+1}\big(SU(r)\big) = n(r-1)$. 
\item The $H$-space $S^7$ is not homotopy-associative \cite{Jam57}, and in particular is not a topological group.  Here we do have the equality $\TC_{n+1}(S^7) = n$ (see \cite[Th.8]{Far03} for $n=1$ and \cite{Rud10} for $n \geq 2$).  But by taking products of $S^7$ with any H-space $Y$ (Lie group or not), \thmref{introthm: TC H-space}  leads to new equalities $\TC_{n+1}(S^7 \times Y) = \cat\big( (S^7\times Y)^n\big)$, where $S^7 \times Y$ is not a homotopy-associative H-space (and in particular is not a topological group).  These equalities lead to many new cases in which $\TC_{n+1}(-)$ may be determined.  For example, we have $\cat\big((S^7)^r\big) = r$, hence $\TC_{n+1}\big((S^7)^r\big) = nr$ from \thmref{introthm: TC H-space}.
\item Arguing as in Example (2), we have new equalities  $\TC_{n+1}\big(S^7 \times U(r)\big) = \cat\big(\big(S^7 \times U(r)\big)^n \big) = n(r+1)$ and $\TC_{n+1}\big(S^7 \times SU(r)\big) = \cat\big(\big(S^7 \times SU(r)\big)^n \big) = nr$.
\end{enumerate}
\end{examples}

Although in general both $\cat(-)$ and $\TC_{n+1}(-)$ behave sub-additively with respect to products, there are situations in which they behave additively.   \thmref{introthm: TC H-space} leads to the following observation.

\begin{corollary}
Suppose $X$ and $Y$ are H-spaces.  For each $n \geq 1$, we have a product equality $\TC_{n+1}(X\times Y) = \TC_{n+1}(X) + \TC_{n+1}(Y)$ if and only if the identity $\cat\big( (X\times Y)^n\big) = \cat(X^n) + \cat(Y^n)$ holds.
\qed\end{corollary}

\thmref{introthm: TC H-space} also leads to other interesting relations involving products and topological complexities, such as the following.

\begin{corollary}
Let $X$ be any CW H-space.  Then for $n \geq 2$ and $k \geq 1$, we have
$$\TC_n(X^k) = \TC_{k+1}(X^{n-1}).$$
\end{corollary}

\begin{proof}
Both are equal to $\cat(X^{k(n-1)})$, by \thmref{introthm: TC H-space}.
\end{proof}

\begin{examples}\label{exs: TC torus}
\cite[Prop.5.1]{Rud10} gives the inequality $\TC_n(T^2) \geq 2n-2$, where $T^2 = S^1 \times S^1$ is the $2$-torus.  (\textbf{Note:}  the $\TC_n$ of  \cite{Rud10} is the number of open sets in a cover, and is thus one greater than our $\TC_n$.)  We may improve on this inequality, as does \cite[Cor.3.13]{BGRT10},  with the formula
$$\TC_n(T^k) = \cat(T^{k(n-1)}) = k(n-1),$$
for $n\geq 2$ and $k\geq 1$, with $T^k = (S^1)^k$ the $k$-torus.  Thus, for example, we have identities $\TC_3(S^1) = \TC(T^2)$ (both equal $2$), and $\TC_4(T^2) = \TC_3(T^3)$ (both equal $6$).  We also observe the sequences $\{\TC_{n+1}(T^k)\}_{n \in \N} = \{k n \}_{n \in \N}$,  as examples of the sequence Rudyak asks after in \cite{Rud10}.  
\end{examples}

We end this section by sketching a variant of the proof we have given for \thmref{introthm: TC H-space}.  We feel the variant is worth mentioning as it uses the following intermediate result, of interest in its own right.

\begin{proposition}\label{prop:difference pullback}
Let $X$ be a connected, CW H-space.  For $n \geq 1$, let $f_n \colon X^{n+1} \to X$ be the map defined in the proof of \thmref{introthm: TC H-space} using the map $D \colon X \times X \to X$.  Then there is a homotopy pullback diagram
$$\xymatrix{X \ar[r] \ar[d]_{\Delta_{n+1}} & {*} \ar[d] \\
X^{n+1} \ar[r]_-{f_n} & X^n.}$$
That is, the diagonal $\Delta_{n+1}\colon X \to X^{n+1}$ may be viewed as the \emph{fibre inclusion} of the map $f_n \colon X^{n+1} \to X^n$.
\end{proposition}

\begin{proof}[Sketch Proof]  One shows that the map $\theta \colon Q \to PX$, as in diagram (\ref{eq: triangle}) in the proof of \thmref{introthm: TC H-space}, is a homotopy equivalence.  
Now by standard results on homotopy pullbacks, we may replace $Q$ in the pullback diagram (\ref{eq: square}) by the homotopy equivalent space $PX$ which, in turn, may be replaced by $X$, and the map $q$ by $P_{n+1}$ and, in turn, $\Delta_{n+1}$.  As usual, we may further replace $\mathcal{P}X$ by a point.
\end{proof}

\begin{proof}[Proof (bis) of \thmref{introthm: TC H-space}]
From the homotopy pullback of \propref{prop:difference pullback}, it follows that $\secat(\Delta_{n+1})\leq \secat(* \to X^n)$, that is, that $\TC_{n+1}(X) \leq \cat(X^n)$. 
\end{proof}

%%%%%%%%%%%%%%%%%%%%%%%%%%%%%%%   
\section{Shear Maps and Homotopy Actions}
\label{sec: Homotopy Transitive}
%%%%%%%%%%%%%%%%%%%%%%%%%%%%%%%   

In this section we generalize \thmref{introthm: TC H-space}.   Our first step in this direction is to generalize the condition that $X$ be an H-space.
Suppose we have a (right) action of a space $Y$ on a space $X$, with strict (right) unit.  That is, we have a map
$$A \colon X \times Y \to X$$
that satisfies $A(x, y_0) = x$, for all $x \in X$ and for $y_0 \in Y$ the basepoint.  Write $\pi_1, \pi_2\colon X \times X \to X$  and $\pi^1\colon X \times Y \to X$ for the projections---preserving previous notation but also distinguishing the distinct domains by our notation.   Then we may define a corresponding \emph{shear map}
$$\sh\colon X \times Y \to X \times X, $$
by $\sh = (\pi^1, A)$, so that $\sh(x,y) = \big(x, A(x,y)\big)$ for $(x,y) \in X \times Y$.  

\begin{definition}
For an action $A \colon X \times Y \to X$ as above, we call a map $D \colon X \times X \to Y$ a \emph{(homotopy) difference (map)} if it satisfies 
$$A \circ (\pi_1, D) \sim \pi_2 \colon X \times X \to X.$$
If an action admits a difference, we say the action is \emph{homotopy-transitive}.
\end{definition}

\begin{remark}
In other words, a homotopy difference map satisfies the identity $A(x, D(x,x')) = x'$ up to homotopy; we think of $D$ as continuously assigning, to each pair $(x,x')$, an element that acts on $x$ so as to ``translate it" to $x'$.   If $X = Y = G$, a topological group, and if the action is the group multiplication, then we may set $D(x,y) = x^{-1}y$.  We used the difference map $D \colon X \times X \to X$ in the proof of \thmref{introthm: TC H-space}, in the context of an H-space $X$ acting on itself by the multiplication.  
\end{remark} 

\begin{lemma}
Let $A \colon X \times Y \to X$ be an action with strict unit and shear map $\sh = (\pi^1, A)\colon X \times Y \to X \times X$.  
The shear map admits a right homotopy inverse if and only if the action admits a difference $D\colon X \times X \to Y$.
\end{lemma}

\begin{proof}
In fact, homotopy classes of  right inverses to the shear map correspond bijectively to homotopy classes of differences, under the identity $[\sigma] = [(\pi_1, D)]$.
Remark that we may write $1_{X \times X} = (\pi_1, \pi_2)\colon X \times X \to X \times X$.  Now suppose given a right inverse $\sigma = (\sigma_1, \sigma_2)$ of $\sh\colon X \times X \to X \times Y$.  Define $D = \sigma_2\colon X \times X \to Y$.  Then we have
\begin{align*}
(\pi_1, \pi_2) = 1_{X \times X} \sim \sh\circ\sigma &= ( \pi^1,  A)\circ(\sigma_1,  D) \sim \big(  \sigma_1, A \circ(\sigma_1, D) \big).
\end{align*} 
Projecting onto first coordinates gives $\pi_1 \sim \sigma_1$, and then projecting onto second coordinates gives $\pi_2 \sim  A\circ(\sigma_1,  D) \sim A\circ(\pi_1,  D)$, so that $D$ is a difference.  Evidently, homotopic right inverses result in homotopic differences.  

Conversely,  suppose given a difference $D \colon X \times X \to Y$, and define $\sigma = (\pi_1, D) \colon$ $X \times X \to X \times Y$.  Then we have $\sh\circ \sigma =  ( \pi^1,  A)\circ(\pi_1,  D) = \big( \pi_1,  A\circ(\pi_1,  D)\big) \sim ( \pi_1,  \pi_2) =  1_{X \times X}$.  Once again, it is evident that homotopic differences result in homotopic right inverses.
\end{proof}

If an action $A \colon X \times Y \to X$ is homotopy-transitive, we may repeat the steps of the proof of \thmref{introthm: TC H-space}, suitably adapted, to obtain inequalities $\TC_{n+1}(X) \leq \cat(Y^n)$ for each $n \geq 1$, and in particular the inequality $\TC(X) \leq \cat(Y)$.  In principle, this appears a satisfactory generalization of \thmref{introthm: TC H-space}. However, the following observation reveals that these upper bounds would actually be weaker than that given by \thmref{introthm: TC H-space}, in all cases in which they apply.

\begin{lemma}\label{lem: Homotopy transitive implies H-space}
Let $A\colon X\times Y \to X$ be a homotopy-transitive action.  Then
\begin{enumerate}
\item $X$ is a retract of $Y$; and 
\item $X$ is an H-space.
\end{enumerate} 
\end{lemma}

\begin{proof}
Let $\sigma = (\pi_1, D)$ be a right inverse of the shear map $(\pi^1, A)$, so that $D\colon X \times X \to Y$ is a difference that satisfies $A\circ (\pi_1,  D) \sim \pi_2$.  Then define $j = D\circ i_2\colon X \to Y$.  Also, let $a = A\circ i_2\colon Y \to X$ be the usual orbit map of the action.  We observe that $i_2\circ D\circ i_2 \sim \sigma\circ i_2\colon X \to X \times X \to X \times Y$, and hence we have
\begin{align*}
a\circ j & \sim (A\circ i_2) \circ (D \circ i_2) = A \circ (i_2\circ D\circ i_2)\\
& \sim (\pi_2\circ \sh) \circ (\sigma\circ i_2) = \pi_2\circ (\sh \circ \sigma) \circ i_2  \sim \pi_2\circ i_2 = 1_X.
\end{align*}
Thus, the orbit map $a \colon Y \to X$ is a retraction of $Y$ onto $X$.

For (2), we define $m = A \circ (1\times  j) \colon X \times X \to X$.  For $m$ to be a multiplication, we only need check that it satisfies the two-sided unit condition: $m\circ i_1 \sim 1_X \sim m\circ i_2 \colon X \to X$.   The second identity is what we have just shown: $a\circ j \sim 1_X$.  The first follows immediately, from the assumptions that $j$ is a based map (which we may assume, without loss of generality, if spaces are well-pointed) and that the action has a strict unit.
\end{proof}

To generalize \thmref{introthm: TC H-space}, then, we go further and also relax the requirement of a section (right homotopy inverse) of the shear map.  We use instead local sections, exactly as in the notion of sectional category.  This leads to local differences, which may be used in the same way as their global counterparts, with suitable adaptations of the argument.  The upshot is the second main result enunciated in \secref{sec:intro}, whose proof we now give.   Notice that the result places no assumption on the action: although the conclusion may not be particularly strong, it has the advantage of applying very generally.

\begin{proof}[Proof of \thmref{introthm: TC secat shear}]
Suppose that $\secat(\sh_{n+1}) = m$, so that we have $m+1$ open sets $\{ U_0, \ldots, U_m \}$ which cover $X^{n+1}$ and for each of which we have a local section of $\sh_{n+1}$.  Fix an $i$, and write the corresponding local section of $\sh_{n+1}$ as $\sigma = (\sigma_1, D_1, \ldots, D_n) \colon U_i \to X \times Y^n$, to define maps $D_j \colon U_i \to Y$.  As in \secref{sec:TC} and the enunciation of \thmref{introthm: TC secat shear} in \secref{sec:intro}, write $p_k\colon X^{n+1} \to X$, $p^1\colon X\times Y^n \to X$, and $p^{k\geq 2}\colon X\times Y^n \to Y$ for the projections.  For the inclusion $\mathrm{incl}\colon U_i \to X^{n+1}$,   
\begin{align*}\mathrm{incl} \sim \sh_{n+1}\circ \sigma &= 
\big( p^1, A\circ(p^1, p^2), \ldots, A\circ(p^1, p^{n+1})  \big)\circ (\sigma_1, D_1, \ldots, D_n) \\ 
&=\big( \sigma_1, A\circ(\sigma_1, D_1), \ldots, A\circ(\sigma_1, D_{n})  \big).
\end{align*}
Equating coordinates yields $\sigma_1 \sim p_1 \colon U_i \to X$, and then $A\circ(p_1, D_{j}) \sim p_{j+1}\colon U_i \to X$ for each $j = 1, \ldots, n$.  In this way, the $D_j$ are thus \emph{local differences} on $U_i$.   
Write $D^i = ( D_1, \ldots,  D_{n}) \colon U_i \to Y^n$, 
and form the pullback
\begin{equation}\label{eq: pullback_i}
\xymatrix{ Q_i \ar[d]_{q_i} \ar[r] & \mathcal{P}Y \ar[d]^{\mathcal{P}_n}\\
U_i \ar[r]_-{D^i} & Y^n, }
\end{equation}
so that 
$$Q_i = \{ (x, x_1, \dots, x_n, \alpha) \in U_i \times \mathcal{P}Y \mid \alpha(0) = y_0, \alpha(t_j) = D_j(x,x_1, \dots, x_n)  \},$$
and $q_i \colon Q_i \to U_i$ is just projection on the first $(n+1)$ coordinates.  
Finally, define a map $\theta_i\colon Q_i \to PX$
by setting $\theta_i\big( (x, x_1, \dots, x_n, \alpha) \big) = A(x, \alpha)$,  which gives a path in $X$ evaluated as $A(x, \alpha)(t) = A\big(x, \alpha(t)\big)$ for each $t \in I$. Write the $n$-tuple $(x_1, \dots, x_n)$ as $\mathbf{x}$, and $(x, x_1, \dots, x_n, \alpha)$ as $(x, \mathbf{x}, \alpha)$.  Then composing $\theta_i$ and $P_{n+1}\colon PX \to X^{n+1}$ gives  
\begin{align*}
P_{n+1}\circ\theta_i (x, \mathbf{x}, \alpha) &= \big(A(x, \alpha(0)), A(x, \alpha(t_1)), \ldots, A(x, \alpha(t_n)) \big)\\
 &= \big(A(x,y_0), A(x, D_1(x,\mathbf{x})), \ldots, A(x, D_n(x,\mathbf{x})) \big)\\ 
 &= \big(p_1, A\circ (p_1, D_1),  \ldots, A(p_1, D_n) \big)\circ q_i (x, \mathbf{x}, \alpha), 
\end{align*}
which gives $P_{n+1}\circ\theta_i \sim \mathrm{incl}\circ q_i$ as maps into $X^{n+1}$, from the argument above diagram (\ref{eq: pullback_i}).  Since $P_{n+1}$ is a fibration, we may replace $\theta_i$ by a homotopic map such that $P_{n+1}\circ\theta_i = \mathrm{incl}\circ q_i$, and thus we have a commutative diagram
$$\xymatrix{ Q_i \ar[d]_{q_i} \ar[r]^-{\theta_i}  & PX \ar[d]^{P_{n+1}}\\
 U_i \ar[r]_-{\mathrm{incl}} & X^{n+1}.}$$
From this diagram, local sections of $q_i$ compose with $\theta_i$ to give local sections of $P_{n+1}$.  Recall that the  steps taken thus far were for a typical $U_i$ in the original $(m+1)$-fold cover of $X^{n+1}$. It follows that we have an upper bound
$$\TC_{n+1}(X) = \secat(P_{n+1})  \leq   \sum_{i=0}^m (\secat(q_i) + 1) - 1.$$
From pullback diagram (\ref{eq: pullback_i}), for each $i$ we have $\secat(q_i) \leq \cat(Y^n)$.  Hence we have    
$$\TC_{n+1}(X)  \leq \big(\secat(\sh_{n+1}) + 1\big) \big(\cat(Y^n) + 1\big) - 1. \qed$$
\renewcommand{\qed}{}\end{proof}

Observe that \thmref{introthm: TC H-space} is indeed a special case of \thmref{introthm: TC secat shear}: in the former case, we have $Y = X$ acting transitively on itself, and $\secat(\sh_{n+1}) = 0$ (the shear map admits a section). 

\begin{remark}\label{rem: cat of difference}
An in-principle finer upper bound on  $\TC_{n+1}(X)$ is obtained by returning to diagram (\ref{eq: pullback_i}) in the proof above and observing that, for each $i$, we have $\secat(q_i) = \cat(D^i)$.  Here, $\cat(D^i)$ denotes the \emph{L-S category of the map} $D^i$; the equality follows from standard results about $\secat(-)$ \cite[Prop.9.18]{CLOT03}.   This yields the upper bound $\TC_{n+1}(X)  \leq \sum_{i=0}^m (\cat(D^i) + 1) - 1$.
\end{remark}

%%%%%%%%%%%%%%%%%%%%%%%%%%%%%%%   
\section{Open Questions}%
\label{sec:relcat}
%%%%%%%%%%%%%%%%%%%%%%%%%%%%%%%   

\subsection{}
Is it possible to characterize spaces $X$ for which $\TC_{n+1}(X)$ takes either its minimum value of $\cat(X^n)$, or its maximum value of $\cat(X^{n+1})$?  We have established the minimum value is taken by $H$-spaces.  Examples at the other extreme include even-dimensional spheres, since we have $\TC_{n+1} (S^{2r}) = n+1$ \cite[Sec.4]{Rud10}, and $\cat\big((S^{2r})^{n+1}\big) = n+1$.
 
\subsection{} \exsref{exs: TC torus} prompt the following:  
For powers $X^k$ of a space $X$, (when)  is there a relation (such as equality) between $\TC_n(X^k)$ and $\TC_{k+1}(X^{n-1})$?  

As regards this question, notice that the standard inequalities 
$$\cat(X^{k(n-1)}) \leq \TC_n(X^k) \leq \cat(X^{kn})$$
and
$$\cat(X^{k(n-1)}) \leq \TC_{k+1}(X^{n-1}) \leq \cat(X^{(n-1)(k+1)})$$
give a common lower bound, but different upper bounds.  Relations of a similar nature are given in \cite[Cor.3.5, Cor.3.6]{BGRT10} 

\subsection{}
From \propref{prop:difference pullback} (cf.~\remref{rem: cat of difference}) it follows that, for $X$ an H-space, we have that $\TC_{n+1}(X) = \cat(f_n)$ (as well as being equal to $\cat(X^n)$).    Are there other situations in which $\TC_{n+1}(X)$ may be identified as the L-S category of some auxiliary map or space?

\subsection{}
We mention briefly another invariant of the L-S type, namely \emph{relative category} (see \cite[Sec.7.2]{CLOT03}).  This is an invariant of a map $p\colon E \to B$, which we denote by $\relcat(p)$.  We omit the definition here, but record the general inequalities $\secat(p) \leq \relcat(p) \leq \cat(B) + 1$.  This invariant would seem of interest as regards $\TC(-)$ because of the resulting inequalities $\TC_{n}(X) \leq \relcat(\Delta_n) \leq \cat(X^n) + 1$.  Is it possible to identify situations in which $\relcat(\Delta)$ gives a better upper bound for $\TC_{n}(X)$ than the usual $\cat(X^n)$?

%\nocite{*}

%\bibliographystyle{amsplain}
%\bibliography{TCRefs}

\providecommand{\bysame}{\leavevmode\hbox to3em{\hrulefill}\thinspace}
\providecommand{\MR}{\relax\ifhmode\unskip\space\fi MR }
\providecommand{\MRhref}[2]{%
  \href{http://www.ams.org/mathscinet-getitem?mr=#1}{#2}
}
\providecommand{\href}[2]{#2}

\end{document}